\documentclass[12pt]{amsart}
\usepackage{fullpage}

\newcommand{\tbigwedge}{\smash{\raisebox{0.2ex}{\ensuremath{\textstyle{\bigwedge}}}}}
\DeclareMathOperator{\Tl}{T}
\DeclareMathOperator{\HH}{H}

\usepackage{amssymb,amsthm,amsmath,amstext,amsxtra}
\usepackage{bm}       
\usepackage{mathtools} 
\usepackage{booktabs}

\usepackage{enumerate}
\usepackage{ stmaryrd }
\usepackage{enumitem}
\usepackage{comment}
\usepackage{colonequals}
\usepackage{numprint}
\usepackage{tikz}
\usepackage{float} 
\usepackage{subcaption}
\usepackage{url}
\usepackage[style=alphabetic, backend=biber, sorting=nyt, doi=false, isbn=false,url=false, hyperref,backref,backrefstyle=none, maxnames=20, maxalphanames=20]{biblatex}
\newbibmacro{string+doiurl}[1]%
{%
  \iffieldundef{doi}%
  {%
    \iffieldundef%
    {url}
    {#1}
    {%
      \href{\thefield{url}}{#1}
    }%
  }%
  {%
    \href{https://doi.org/\thefield{doi}}{#1}%
  }%
}

\DeclareFieldFormat{title}{\usebibmacro{string+doiurl}{\mkbibemph{#1}}}
\DeclareFieldFormat[article,incollection,inproceedings,unpublished,misc,book]{title}%
    {\usebibmacro{string+doiurl}{\mkbibquote{#1}}}
\DefineBibliographyStrings{english}{%
  backrefpage = {$\uparrow$},%
  backrefpages = {$\uparrow$}%
}
\usepackage{url}
\usepackage{hyperref}
\hypersetup{colorlinks=true,urlcolor=blue,citecolor=blue,linkcolor=blue}

\numberwithin{equation}{section}

\usepackage{thmtools}
\usepackage[nameinlink]{cleveref}

\theoremstyle{plain}
\newtheorem*{theorema*}{Theorem A}

\newtheorem{thm}[equation]{Theorem}

\newtheorem{lem}[equation]{Lemma}

\theoremstyle{definition}

\theoremstyle{remark}
\newtheorem{remark}[equation]{Remark}

\newenvironment{enumalph}
{\begin{enumerate}}
{\end{enumerate}}

\newcommand{\Z}{\mathbb{Z}}
\newcommand{\Q}{\mathbb{Q}}

\let\C\relax
\newcommand{\C}{\mathbb{C}}

\newcommand{\PP}{\mathbb{P}}

\DeclareMathOperator{\cond}{cond}
\DeclareMathOperator{\GO}{GO}
\DeclareMathOperator{\Jac}{Jac}

\DeclareMathOperator{\Ind}{Ind}

\DeclareMathOperator{\Galois}{Gal}
\newcommand{\Gal}[2]{\Galois({#1}\,|\,{#2})}

\newcommand{\et}{{\mathrm{\acute{e}t}}}

\DeclareMathOperator{\rk}{\mathrm{rk}}

\DeclareMathOperator{\NS}{\mathrm{NS}}

\DeclareMathOperator{\Pic}{\mathrm{Pic}}

\DeclareMathOperator{\alg}{{al}}
\newcommand{\al}{{\alg}}

\newcommand{\frakp}{\mathfrak{p}}

\newcommand{\frakN}{\mathfrak{N}}

\newcommand{\Fbar}{{F^\al}}

\usepackage{xcolor}
\definecolor{darkred}{HTML}{CC1F1F}
\definecolor{green}{rgb}{.4,.7,.4}
\definecolor{blue}{rgb}{.2,.6,.75}
\definecolor{pastelb}{HTML}{3333FF}

\definecolor{pastelyellow}{rgb}{0.992157, 0.552941, 0.235294}
\definecolor{pastelorange}{rgb}{0.941176, 0.231373, 0.12549}
\definecolor{pastelred}{rgb}{0.741176, 0., 0.14902}
\definecolor{darkbrown}{rgb}{0.25098, 0., 0.0745098}

\usepackage{hyperref}
\hypersetup{pdftitle={Explicit modularity of K3 surfaces with complex multiplication of large degree},
pdfauthor={Edgar Costa, Andreas-Stephan Elsenhans, J\"org Jahnel, John Voight}}
\hypersetup{colorlinks=true,linkcolor=blue,anchorcolor=blue,citecolor=blue}

\usepackage{arydshln} 

\begin{document}

\title[Explicit modularity of K3 surfaces with CM of large degree]{Explicit modularity of K3 surfaces \\ with complex multiplication of large degree}

\author[Costa]{Edgar Costa}
\address{Department of Mathematics, Massachusetts Institute of Technology, 77 Massachusetts Ave., MA 02139, USA}
\email{edgarc@mit.edu}
\urladdr{\url{https://edgarcosta.org}}

\author[Elsenhans]{Andreas-Stephan Elsenhans}
\address{School of Mathematics and Statistics, Carslaw Building (F07), University of Sydney, NSW 2006, Australia}
\address{Institut f\"ur Mathematik \\ Emil-Fischer Stra\ss e 30 \\ D-97074 W\"urzburg, Germany}
\email{Stephan.Elsenhans@uni-wuerzburg.de}
\urladdr{\url{https://www.mathematik.uni-wuerzburg.de/computeralgebra/team/elsenhans-stephan-prof-dr/}}

\author[Jahnel]{J\"org Jahnel}
\address{\mbox{Department Mathematik\\ Univ.\ \!Siegen\\ \!Walter-Flex-Str.\ \!3\\ D-57068 \!Siegen\\ \!Germany}}
\email{jahnel@mathematik.uni-siegen.de}
\urladdr{\url{http://www.uni-math.gwdg.de/jahnel}}

\author[Voight]{John Voight}
\address{Department of Mathematics, Dartmouth College, 6188 Kemeny Hall, Hanover, NH 03755, USA}
\address{School of Mathematics and Statistics, Carslaw Building (F07), University of Sydney, NSW 2006, Australia}
\email{jvoight@gmail.com}
\urladdr{\url{https://jvoight.github.io/}}


\begin{abstract}
We consider the transcendental motives of three K3 surfaces $X$ conjectured to have complex multiplication (CM).
Under this assumption, we match these to explicit algebraic Hecke quasi-characters $\psi_X$, and CM abelian threefolds $A$.
This provides substantial evidence that a power of $A$ corresponds to $X$ under the Kuga--Satake correspondence.
\end{abstract}


\maketitle


\section{Introduction}

K3 surfaces provide a rich class of objects to study in number theory and the Langlands program, testing conjectures that connect arithmetic geometry and automorphic forms through Galois representations and $L$-functions.

The case where the Picard number $\rho$ achieves its maximum $\rho=20$ has been well-studied.  Potential modularity was established through their association with algebraic Hecke quasi-characters (also called Hecke Grossencharacters or just Hecke characters) by Shioda--Inose \cite[\S 6, Theorem 6]{shioda-inose-77} (see also Livn\'e \cite{livne-95}): the transcendental cohomology has complex multiplication (CM) by an imaginary quadratic field.  Over $\Q$, an explicit correspondence with classical modular forms of weight $3$ was worked out by Elkies--Sch\"{u}tt \cite{elkies-schuett-13}.

Recent efforts towards incorporating K3 surfaces into the \emph{$L$-functions and Modular Forms Database (LMFDB)} \cite{lmfdb} has renewed questions of explicit modularity for K3 surfaces, but less is known about modularity for K3 surfaces of lower Picard number.  In general, a complex K3 surface $X$ with $\rho(X) \leq 16$ does not admit a Shioda--Inose structure.  Piatetski-Shapiro--Shafarevich \cite{MR335521} expressed the $L$-function of a K3 surfaces with complex multiplication as a product of Hecke $L$-functions over some finite extension via the Kuga--Satake correspondence and applying the corresponding statement for abelian varieties.  The theory of complex multiplication for K3 surfaces was further developed by Rizov \cite{rizov}.  Building on this work, Valloni \cite{valloni-23} considers K3 surfaces with CM by the full ring of integers and studied their fields of definition; and more recently, Ito \cite{ito2025motivicmodularitycmk3} is more explicit about the properties of the Hecke quasi-characters that appear in the equality of $L$-series.  Livn\'e--Sch\"utt--Yui \cite{livne-schuett-yui-10} established modularity for the finitely many K3 Delsarte surfaces (up to twist): they are (CM) quotients of Fermat surfaces. 

%
%
%

This paper advances these efforts in a new direction, through computation.
We remain focused on explicit examples of K3 surfaces with apparent CM of large degree.  Indeed, there has been recently renewed interest \cite{bayer-geemen-schuett-24} in moduli of K3 surfaces with extra Hodge endomorphisms.
Our main result matches the transcendental cohomology of certain K3 surfaces with algebraic Hecke quasi-characters, as follows.
For a complex surface $X$, we let $\Tl(X)_\Q \subseteq \HH^2(X,\Q)$ be the transcendental subspace (see \cref{sec:setup}).
If moreover $X$ is defined over a number field $F \subset \C$, then for $\ell$ prime, we have via comparison $\Tl(X) \otimes \Q_\ell \hookrightarrow \HH^2_{\textup{\'et}}(X,\Q_\ell)$ and we let $\rho_{\Tl(X),\ell} \colon \Galois_{F} \circlearrowright (\Tl(X)\otimes{\Q_\ell})$ be the associated Galois representation.

Let $X=X_i$ for $i=1,2,3$ be one of the three K3 surfaces obtained from the following affine models:
\begin{equation}    \label{eqn:defnX}
  \begin{aligned}
    X_1 \colon w^2 &= x y z (x^3 - 3xy^2 + y^3 - 3x^2z - 3xyz + 9y^2z + 6yz^2 + z^3) \\
    X_2 \colon w^2 &= x y z (7x^3 - 7x^2y + y^3 + 49x^2z - 21xyz - 7y^2z + 98xz^2 + 49z^3) \\
    X_3 \colon w^2 &= x y z \left(
      \text{
\begin{tabular}{@{}c@{}}
  $49x^3 - 304x^2y + 361xy^2 + 361y^3 + 570x^2z  - 2793xyz$  \\  $+ 2888y^2z + 2033xz^2 - 5415yz^2 + 2299z^3$
\end{tabular}}
  \right).
  \end{aligned}
\end{equation}
More precisely, we take $X_i$ to be the smooth projective surface obtained from the taking branched double cover of $\PP^2$ defined by \eqref{eqn:defnX} and blowing up the $15=\binom{6}{2}$ double points in the branch locus of $6$ lines.  Then $\dim_{\Q_\ell} \Tl(X_i)=22-16=6$, and there is substantial numerical evidence that in each case, $\Tl(X_i)_{\Q}$ has CM by $K_i$, where $K_i=F_i(\sqrt{-1})$ is the cyclic sextic field defined in \Cref{table:defnKandC} by their LMFDB label.  For this evidence, see Elsenhans--Jahnel \cite[\S 5]{elsenhans-jahnel-16} and the end of \cref{sec:setup}.  

\begin{thm} \label{thm:mainthm}
  For $i=1,2,3$ and $X=X_i$, the following statements hold.
  \begin{enumalph}
    \item Suppose $\Tl(X)_\Q$ has CM by $K$.
      Then for all primes $\ell$,
      \begin{equation}
        \rho_{\Tl(X),\ell} \simeq \Ind_{\Galois_K}^{\Galois_\Q} \psi_{X, \ell}
      \end{equation}
      where $\psi_X$ is of $\infty$-type $\{(0,2),(1,1),(1,1)\}$ defined in \textup{\Cref{table:psiX}}.
      In particular, we have 
      \begin{equation}
      L(\Tl(X),s)=L(s,\psi_X).
      \end{equation}

    \item Let $A=A_i=\Jac(C_i)$ be the Jacobian defined in \textup{\Cref{table:defnKandC}}.
      Then
      \begin{equation}
      \begin{aligned}
        \rho_{\HH^1(A),\ell} &\simeq \Ind_{\Galois_K}^{\Galois_\Q} \psi_{A, \ell} \\
        L(\HH^1(A),s) &=L(s,\psi_A)
        \end{aligned}
      \end{equation}
      where $\psi_A$ is of $\infty$-type $\{(0,1),(0,1),(0,1)\}$ defined in \textup{\Cref{table:psiA}}.
    \item We have
      \begin{equation}
        \begin{aligned}
          \rho_{\HH^2(A),\ell} &\simeq \Ind_{\Galois_F}^{\Galois_\Q} \Q_\ell(1) \oplus \Ind_{\Galois_K}^{\Galois_\Q} (\psi_X \oplus \psi')_\ell
          \\
          L(\HH^2(A),s) &= \zeta_F(s+1) L(s,\psi_X) L(s,\psi'),
        \end{aligned}
      \end{equation}
      where $\Q_\ell(1)$ is the Tate twist, $F \subseteq K$ is the unique cubic subfield, and $\psi'$ is of $\infty$-type $\{(0,2),(0,2),(1,1)\}$ defined in \textup{\Cref{table:psi'}}.
  \end{enumalph}
\end{thm}

This provides substantial evidence that a power of $A$ corresponds to $X$ under the Kuga--Satake correspondence \cite{kuga-satake-67}: for more, see \Cref{rmk:kugasatake}.  
Our computations are performed in \textsf{Magma} \cite{magma}; the code is available at \url{https://github.com/edgarcosta/K3withCM/}.  There is a natural Galois action on algebraic Hecke quasi-characters by $\psi^\sigma = \psi \circ \sigma$ for $\sigma \in \Gal{K}{\Q}$, with $L(s,\psi^\sigma)=L(s,\psi)$. 
Up to this Galois action, the characters $\psi_X$, $\psi_A$, and $\psi'$ in \Cref{thm:mainthm} are unique.

\begin{remark}
To each Hecke quasi-character $\psi$ for a CM extension $K \supseteq F$, by restriction of the automorphic representation we can also associate a Hilbert modular form $f$ over $F$ with matching Galois representation and $L$-function.  As such forms $f$ have nontrivial character (and in parts (a) and (c), weights $(3,3,1)$ and $(3,1,1)$, respectively), they currently fall outside the database of Hilbert modular forms in the LMFDB.  We hope to see them in a future expansion of the database.
\end{remark}


\Cref{table:defnKandC} comes from Weng \cite[\S6]{weng-01} and is certified correct \cite{costa-mascot-sijsling-voight-18}.
For the fourth row ($i=4$), we were not able to find a matching K3 surface (among double covers of $\PP^2$ branched along $6$ lines, possibly due to the nontrivial class group), but part (b) still holds; it would be interesting to produce a K3 surface in this case (not necessarily a degree 2 model).

\subsection*{Acknowledgements}
We thank Eran Assaf and David Roe for helpful conversations and Eva Bayer-Fluck\-iger and the anonymous referees for corrections and suggestions.  Costa (SFI-MPS-Infra\-structure-00008651, AS) and Voight (SFI-MPS-Infra\-structure-\-00008650, JV) were supported by Simons Foundation grants.

\begin{table}\renewcommand{\arraystretch}{1.1}
  \begin{tabular}{c|ccc}
    $i$ & $K_i$ & $F_i$ & Defining equation for $C_i$ \\
    \hline\hline
$1$ & \href{https://www.lmfdb.org/NumberField/6.0.419904.1}{\textsf{6.0.419904.1}} 
& \href{https://www.lmfdb.org/NumberField/3.3.81.1}{\textsf{3.3.81.1}}
& $y^2 =x^7 + 6x^5 + 9x^3 + x$\\

$2$ & \href{https://www.lmfdb.org/NumberField/6.0.153664.1}{\textsf{6.0.153664.1}} & 
\href{https://www.lmfdb.org/NumberField/3.3.49.1}{\textsf{3.3.49.1}} &
$y^2 =x^7 + 7x^5 + 14x^3 + 7x$\\
$3$ & \href{https://www.lmfdb.org/NumberField/6.0.8340544.1}{\textsf{6.0.8340544.1}} & 
\href{https://www.lmfdb.org/NumberField/3.3.361.1}{\textsf{3.3.361.1}} &
$y^2 =x^7 + 1786x^5 + 44441x^3 + 278179x$\\
\hdashline
$4$ & \href{http://www.lmfdb.org/NumberField/6.0.59105344.1}{\textsf{6.0.59105344.1}} &
\href{http://www.lmfdb.org/NumberField/3.3.961.1}{\textsf{3.3.961.1}} &
$y^2 =x^7 + 961x^5 - 3694084x^3 + 1832265664x$
  \end{tabular}\vspace{-2ex}
  \caption{Polynomials defining CM fields and genus 3 curves.}
  \label{table:defnKandC}
\end{table}

\begin{table}\renewcommand{\arraystretch}{1.1}
    \begin{tabular}{r|cccc}
    $i$ & $\cond(\psi_{X_i})$ & $p$ & $L_p(\psi_{X_i}, T)$\\
    \hline\hline
$1$ & \textsf{64.1} & $17$ & $1 - 6   T + 15   p T^2 + 12    {p}^{2} T^{3} + 15    {p}^{3} T^{4} - 6    {p}^{4} T^{5} +  {p}^{6} T^{6}$\\
$2$ & \textsf{3136.1} & $13$ & $1 - 2   T + 19   p T^2 + 4    {p}^{2} T^{3} + 19    {p}^{3} T^{4} - 2    {p}^{4} T^{5} +  {p}^{6} T^{6}$\\
$3$ & \textsf{23104.1} & $37$ & $1 + 14   T - 5   p T^2 - 28    {p}^{2} T^{3} - 5    {p}^{3} T^{4} + 14    {p}^{4} T^{5} +  {p}^{6} T^{6}$\\
\hdashline 
$4$ & \textsf{61504.13} & $29$ & $1 - 38   T - 9   p T^2 + 52    {p}^{2} T^{3} - 9    {p}^{3} T^{4} - 38    {p}^{4} T^{5} +  {p}^{6} T^{6}$
  \end{tabular}\vspace{-2ex}
  \caption{Uniquely defining properties of $\psi_X$, up to $\Gal{K}{\Q}$.}
  \label{table:psiX}
\end{table}

\begin{table}\renewcommand{\arraystretch}{1.1}
    \begin{tabular}{r|cccc}
    $i$ & $\cond(\psi_{A_i})$ & $p$ & $L_p(\psi_{A_i}, T)$\\
    \hline\hline
$1$ & \textsf{4096.1} & $17$ & $1 - 6   T + 15    T^{2} - 52    T^{3} + 15   pT^4 - 6    {p}^{2} T^{5} +  {p}^{3} T^{6}$\\
$2$ & \textsf{25088.1} & $13$ & $1 + 4   T + 7    T^{2} + 40    T^{3} + 7   pT^4 + 4    {p}^{2} T^{5} +  {p}^{3} T^{6}$\\
$3$ & \textsf{184832.1} & $37$ & $1 + 4   T + 15    T^{2} - 152    T^{3} + 15pT^4    + 4    {p}^{2} T^{5} +  {p}^{3} T^{6}$\\
\hdashline 
$4$ & \textsf{3936256.41} & $29$ & $1 + 4   T + 51    T^{2} + 216    T^{3} + 51pT^4   + 4    {p}^{2} T^{5} +  {p}^{3} T^{6}$
  \end{tabular}\vspace{-2ex}
  \caption{Uniquely defining properties of $\psi_A$, up to $\Gal{K}{\Q}$.}
  \label{table:psiA}
\end{table}

\begin{table}\renewcommand{\arraystretch}{1.1}
  \begin{tabular}{r|cccc}
    $i$ & $\cond(\psi'_i)$ & $p$ & $L_p(\psi'_i, T)$\\
    \hline\hline
$1$ & \textsf{1.1} & $17$ & $1 + 42T + 1023T^{2} + 1132pT^3 + 1023p^2T^{4} + 42p^4 T^{5} +  p^6 T^{6}$ \\
$2$ & \textsf{1.1} & $13$ & $1 + 34   T + 631    T^{2} + 652   p T^3 + 631    {p}^{2} T^{4} + 34    {p}^{4} T^{5} +  {p}^{6} T^{6}$\\
$3$ & \textsf{1.1} & $37$ & $1 + 82   T + 4423    T^{2} + 5452   pT^3 + 4423    {p}^{2} T^{4} + 82    {p}^{4} T^{5} +  {p}^{6} T^{6}$\\
\hdashline 
$4$ & \textsf{1.1} & $29$ & $1 + 74   T + 3067    T^{2} + 3268   pT^3 + 3067    {p}^{2} T^{4} + 74    {p}^{4} T^{5} +  {p}^{6} T^{6}$
\end{tabular}
  \caption{Uniquely defining properties of $\psi'$, up to $\Gal{K}{\Q}$.}
  \label{table:psi'}
\end{table}

\section{Setup} \label{sec:setup}

Let $X$ be a polarized K3 surface over a number field $F$.
We denote by $X^\al$ its base change to $F^\al$, an algebraic closure of $F$.
We are interested in studying the Galois representations that arise from $\HH^2 _\et (X^\al, \Q_\ell)$, for a prime $\ell$.
Let $\NS(X)$ denote the N\'{e}ron--Severi group of $X$.
Under the canonical isomorphism $\NS(X) \cong \Pic(X)$, we may identify $\NS(X^\al) \cong \HH^2(X_\C, \Z) \cap \HH^{1,1}(X, \C) \subsetneq \HH^2(X_\C, \Z)$.  Let $\rho(X) \colonequals \rk \NS(X)$ be the Picard number.

Let $\Tl(X)$ be the transcendental lattice of $X$, the orthogonal complement of $\NS(X_{\C})$ in $\HH^2(X, \Z)$.
The space $\Tl(X)$ is a sub-Hodge structure of $\HH^2(X_\C, \Z)$ with Hodge numbers $(1, 20 - \rho(X), 1)$.
Let $E=E(X)$ be the algebra of endomorphisms of $\Tl(X)$ that respect the Hodge structure.
Zarhin \cite[Theorems 1.5.1, 1.6]{zarhin-83} has shown that $E$ is either a totally real field or a CM field.

The Galois representation
$
\rho_{\HH^2, \ell} \colon \Gal{\Fbar}{F} \rightarrow \GO( \HH^2(X_\C, \Z) \otimes \Q_\ell )
$
decomposes as $\rho_{\HH^2, \ell} = \rho_{\NS, \ell} \oplus \rho_{\Tl, \ell}$, and we focus on
\begin{align*}
\rho_{\Tl, \ell} \colon & \Gal{\Fbar}{F}  \rightarrow \GO(\Tl(X) \otimes \Q_\ell).
\end{align*}
and its associated \cite{serre-70} $L$-function $L(\Tl(X),s)$.  In the case that $\dim_{E} \Tl(X) = 1$, in fact $E$ is necessarily a CM field, and by class field theory we have $L(\Tl(X), s) = L(s,\psi)$ for some algebraic Hecke quasi-character $\psi$ over $E$.

We consider K3 surfaces $X \rightarrow \PP^2$ as (resolutions of) branched over $6$ lines in general (and in particular in good) position.  In this case, $\rho(X) \geq 16$, with equality when the lines are in very general position.  
We restrict to the K3 surfaces identified in \eqref{eqn:defnX}.  As mentioned in the introduction, there is strong numerical evidence \cite[\S 5]{elsenhans-jahnel-16} that these K3 surfaces have complex multiplication (CM).  We further computed 100 digit approximations to the period lattices using the method of Elsenhans--Jahnel \cite[\S 6]{elsenhans-jahnel-18}.  In fact, this CM is apparently by the maximal order $\Z_{K_i}$ in each case. More precisely, for each surface we found numerical approximations of six period integrals $\tau_1,\ldots, \tau_6$
that form a basis of the period lattice such that the ratios $\tau_i/\tau_1$ for $i=1,\dots,6$ coincide with a $\Z$-basis of the maximal order of the conjectural endomorphism field.
For the surface $X_1$ and for chosen cycles,
$$(\tau_1,\ldots,\tau_6) \approx
( 2.6402, 11.6474, 7.60232 , -7.6023  i, -4.96206 i, -6.68537  i );
$$
with respect to the eigenvalues $0.467911$ and $i$, the period vector is an eigenvector of
\scalebox{0.6}{$
\left(
\begin{array}{cccccc} 
2  &  -1  &  1  &    &    &    \\ 
 -1  &  2  &  -2  &    &    &    \\ 
 0  &  -1  &  2  &    &    &    \\ 
   &    &    &  2  &  -1  &  -1  \\ 
   &    &    &  -1  &  2  &  0  \\ 
   &    &    &  -2  &  1  &  2  
\end{array} 
\right), ~
\left(
\begin{array}{cccccc} 
   &    &    &  -1  &  1  &  0  \\ 
   &    &    &  0  &  -1  &  -1  \\ 
   &    &    &  -1  &  0  &  0  \\ 
 0  &  0  &  1  &    &    &    \\ 
 -1  &  0  &  1  &    &    &    \\ 
 1  &  1  &  -1  &    &    &    \\ 
\end{array}
\right)$}
and
the cup form is 
\scalebox{0.6}{$\left(
\begin{array}{cccccc}
-1 & 0 & 1 &  &  &  \\
 0 & 0 & 1 &  &  &  \\
 1 & 1 &-1 &  &  &  \\
  &  & & 0 & 1 & 1 \\
  &  &  & 1 &-1 & 0 \\
  &  &  & 1 & 0 & 0 \\
\end{array}
\right)$}.
The data for the other examples are very similar.  Computing periods to a precision of 100 decimal places took about half an hour on a standard desktop.  
 
\section{Proof of main result}

Under the assumption that the $L$-function matches an algebraic Hecke quasi-character, to find the correct one we need to bound its conductor.  It seems difficult in general to obtain such a bound by computing the conductor of the $L$-function of the K3 surface.
We can however produce a finite list of possibilities as follows.
We start with the list of bad primes of the K3 surface and the primes above them in $K$. 
To bound the exponents of these primes, recall that (by the $\mathfrak{p}$-adic logarithm) the unit groups $(\Z_{K,\frakp}/\frakp^e)^\times$ as $e \to \infty$ have a bounded number of invariant factors.  So to show that the exponent is bounded, we just need to show that the order of the finite part of the Hecke quasi-character is bounded, using the following lemma.

\begin{lem} \label{lemma:bound}
Let $\psi$ be an algebraic Hecke quasi-character over $K$ of modulus $\frakN$ and let $M \subset \C$ be the field generated by the values of $\psi$.  Let $\chi \colon (\Z_K/\frakN)^\times \to \C^\times$ be the Dirichlet character defined by $\chi(a)=\psi(a\Z_K)\psi_\infty(a)$.  Then $\Q(\chi) \subseteq M$.
\end{lem}


\begin{proof}
By definition, an algebraic Hecke quasi-character takes values in a number field.  From the idelic formulation, we conclude that the subfield generated by the restriction of $\psi$ to the infinite places is contained in $M$, hence also $\Q(\chi)$.
\end{proof}


\begin{proof}[Proof of \Cref{thm:mainthm}]
We first prove (a).
We compute the bad primes for $X$ by checking if the reduction no longer leads to 15 distinct intersection points of the 6 lines. 
We bound the exponents of the primes using \Cref{lemma:bound}. 
Following Watkins \cite[\S~5.2]{watkins}, using \textsf{Magma} we compute the full list of algebraic Hecke quasi-characters $\psi$ with the required $\infty$-type, conductor bounded as above, and $\Q(\psi) \subseteq K_i$.
More precisely, we start with the principal character $\psi_0$ of the chosen $\infty$‑type and its associated Dirichlet character $\chi_0$ (see \Cref{lemma:bound}).  
Next, we enumerate the Dirichlet characters $\chi$ whose lifts to Hecke characters twist $\psi_0$ to give a primitive character $\psi$ with $\Q(\psi)\subseteq K_i$. 
Concretely, we require that $\chi' \colonequals \chi/\chi_0$
be primitive, trivial on units, and satisfy $\Q(\chi)\subseteq K_i$.
Because all these conditions can be phrased on the abstract character group, we apply the filters there----iterating over every element would be impractical for large levels.  For example, for $X_3$ we consider characters of conductor  $\frakN = \frakp_2^7 \cdot 7 \cdot 11 \cdot \frakp_{19}$, where $\frakp_p$ is the unique prime dividing $p$.
The Dirichlet character group modulo $\frakN$ is isomorphic to
\[(\Z/4 \Z)^5 \oplus (\Z/8 \Z)^2 \oplus (\Z/24\Z)^2 \oplus \Z/ 48 \Z \oplus (\Z/240 \Z)^3 \oplus \Z/5040\Z\]
which contains over $10^{20}$ elements; of these, only \numprint{279\,936} satisfy our requirements.

We then compute $L_p,(\Tl(X),T)$ using a method of Elsenhans--Jahnel \cite{elsenhans-jahnel-16} based on a trace formula involving a matrix expansion.
In the style of Sherlock Holmes, we eliminate all but one (up to the action of $\Gal{K}{\Q}$) by finding primes $p$ uniquely identifying $L_p(\Tl(X),T)=L_p(\psi,T)$. 
It was enough to consider good primes $p < 250$ totally split in $K_i$ to obtain a unique match for each example in a single pass.  After a match was found, we identified a prime $p$ which, together with the conductor, uniquely identifies the character (up to the Galois action).

Part (b) is proven in the same way as part (a).  For part (c), we note that $\HH^2(A) \simeq \tbigwedge^2 \HH^1(A)$, so applying (b) and identifying characters we find that $\HH^2(A) \simeq V_1 \oplus V_2 \oplus V_3$ as representations of $\Galois_{\Q}$, where $V_1 \simeq \Ind_{\Galois_F}^{\Galois_\Q} \Q_\ell(1)$ and $\dim_K V_2=\dim_K V_3=1$.  The relationship between the two characters $\psi_A$ and $\psi_X$ can be further encoded by the equality $\psi_X(\mathfrak{p}) = \psi_A(\sigma_1(\mathfrak{p}))\psi_A(\sigma_2(\mathfrak{p}))$
for all unramified primes $\mathfrak{p}$ of degree $1$ and where $\sigma_1,\sigma_2 \in \mathrm{Gal}(K\,|\,\mathbb{Q})$ are the two elements of order $3$.  We finish as in (a), checking on distinguishing primes.  
\end{proof}

\begin{remark} \label{rmk:kugasatake}
The Kuga--Satake construction \cite{kuga-satake-67} (see also van Geemen \cite[\S 5]{van-geemen-00}) attaches to a complex polarized K3 surface $X$ a complex abelian variety such that there is an embedding $\Tl(X)(1) \hookrightarrow H^1(A) \otimes H^1(A)$ of Hodge structures.  In our case, this relationship is made explicit in \Cref{thm:mainthm}(c) via comparison, in the sense that $X$ and $A$ have associated to Hecke characters $\psi_X$ and $\psi_A$, with $\psi_X$ appearing as a symmetric product of $\psi_A$.  This strongly suggests that $X$ and $A$ are connected via the Kuga--Satake construction, at least up to isogeny and powers.  
\end{remark}

\printbibliography

\end{document}